\documentclass[10pt]{article}
\font\smallit=cmti10

\usepackage{latexsym,epsfig} 
\usepackage{amsmath}
\usepackage{amsthm, amssymb}
\usepackage{mathtools}
\usepackage{graphicx}   
\usepackage{url}
\usepackage[utf8]{inputenc} 
\usepackage{mathrsfs} 
\usepackage{ytableau} 
\usepackage{tikz}
\usetikzlibrary{arrows}
\definecolor{lgray}{rgb}{0.8,0.8,0.8}

\makeatletter

\renewcommand\section{\@startsection {section}{1}{\z@}
{-30pt \@plus -1ex \@minus -.2ex}
{2.3ex \@plus.2ex}
{\normalfont\normalsize\bfseries}}

\renewcommand\subsection{\@startsection{subsection}{2}{\z@}
{-3.25ex\@plus -1ex \@minus -.2ex}
{1.5ex \@plus .2ex}
{\normalfont\normalsize\bfseries}}

\renewcommand{\@seccntformat}[1]{\csname the#1\endcsname. }

\makeatother

\newtheorem{theorem}{Theorem}
\newtheorem{lemma}{Lemma}
\newtheorem{conjecture}{Conjecture}
\newtheorem{corollary}{Corollary}
\newtheorem{observation}{Observation}
\newtheorem{definition}{Definition}

\newcommand{\rescaled}[1]{\partial{#1}}

\newcommand{\Rnn}{\mathbb{R}_{\geq 0}}
\newcommand{\ceilz}{\lceil z \rceil}

\begin{document}

\begin{center}
\uppercase{\bf Limit shapes of stable configurations of a generalized Bulgarian solitaire}
\vskip 20pt
{\bf Kimmo Eriksson}\\
{\smallit Mälardalen University, School of Education, Culture and Communication,\\Box 883, SE-72123 Västerås, Sweden}\\
{\tt kimmo.eriksson@mdh.se}\\
\vskip 10pt
{\bf Markus Jonsson\footnote{Corresponding author}}\\
{\smallit Mälardalen University, School of Education, Culture and Communication,\\Box 883, SE-72123 Västerås, Sweden}\\
{\tt markus.jonsson@mdh.se}\\
\vskip 10pt
{\bf Jonas Sjöstrand}\\
{\smallit Royal Institute of Technology, Department of Mathematics,\\SE-10044 Stockholm, Sweden}\\
{\tt jonass@kth.se}\\
\end{center}
\vskip 30pt

\centerline{\smallit \today}

\vskip 30pt

\centerline{\bf Abstract}

\noindent
Bulgarian solitaire is played on $n$ cards divided into several piles; a move consists of picking one card from each pile to form a new pile. In a recent generalization, $\sigma$-Bulgarian solitaire,  the number of cards you pick from a pile is some function $\sigma$ of the pile size, such that you pick $\sigma(h)\le h$ cards from a pile of size $h$. Here we consider a special class of such functions. Let us call $\sigma$ well-behaved if $\sigma(1)=1$ and if both $\sigma(h)$ and $h-\sigma(h)$ are non-decreasing functions of $h$. Well-behaved  $\sigma$-Bulgarian solitaire has a geometric interpretation in terms of layers at certain levels being picked in each move. It also satisfies that if a stable configuration of $n$ cards exists it is unique. Moreover, if piles are sorted in order of decreasing size ($\lambda_1 \ge \lambda_2\ge \dots$) then a configuration is convex if and only if it is a stable configuration of some well-behaved  $\sigma$-Bulgarian solitaire. If sorted configurations are represented by Young diagrams and scaled down to have unit height and unit area, the stable configurations corresponding to an infinite sequence of well-behaved functions ($\sigma_1, \sigma_2, \dots$) may tend to a limit shape $\phi$. We show that every convex $\phi$ with certain properties
can arise as the limit shape of some sequence of well-behaved $\sigma_n$. For the special case when $\sigma_n(h)=\lceil q_n h \rceil$ for $0 < q_n \le 1$, these limit shapes are triangular (in case $q_n^2 n\rightarrow 0$), or exponential (in case $q_n^2 n\rightarrow \infty$), or interpolating between these shapes (in case $q_n^2 n\rightarrow C>0$).

\pagestyle{myheadings}
\thispagestyle{empty}
\baselineskip=12.875pt
\vskip 30pt

\section{Introduction}
\label{sec:bulgsol}
The game of Bulgarian solitaire is played with a deck of $n$ identical cards divided arbitrarily into several piles. A move consists of picking a card from each pile and letting these cards form a new pile. This move is repeated over and over again. For  information about the earlier history of the Bulgarian solitaire and a summary of subsequent research, see reviews by Hopkins \cite{Hopkins2011} and Drensky \cite{drensky2015bulgarian}. 

Let $\mathcal{P}(n)$ denote the set of integer partitions of $n$. An integer partition of $n$ is a $\lambda = (\lambda_1,\lambda_2,\dotsc,\lambda_\ell)$ such that $\lambda_1\geq\lambda_2\geq\ldots\geq\lambda_\ell>0$ and $\sum_{i=1}^\ell \lambda_i=n$. For $i>\ell(\lambda)$ it will be convenient to define $\lambda_i=0$. The number of non-zero parts of the partition $\lambda$ is denoted by $\ell=\ell(\lambda)$. If piles of cards are sorted in order of decreasing size, any configuration of $n$ cards can be regarded as an integer partition of $n$. A geometric shape arises when a configuration $\lambda$ is represented by a Young diagram of unit squares in the first quadrant of a coordinate system for the real plane, such that the $i$th column has height $\lambda_i$.  A move of the Bulgarian solitaire then has the geometric interpretation of picking the first (i.e., bottom) layer of the diagram and making it the new first column, left-shifting cards if needed so that the configuration remains sorted.

\ytableausetup{mathmode, boxsize=13pt, aligntableaux=center}
\begin{figure}[ht]
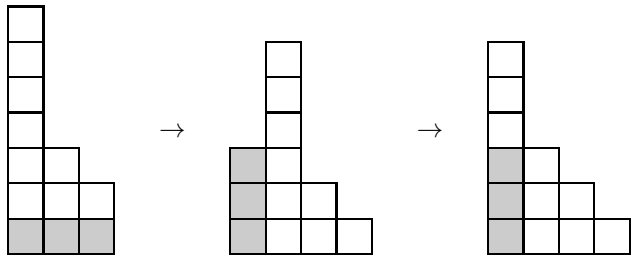

	\centering
	\begin{ytableau}
		\none & \\
		\none & \\
		\none & \\
		\none & \\	
		\none & & \\
		\none & & & \\
		\none & *(lgray) & *(lgray) & *(lgray) & \none
	\end{ytableau}
	$\to$
	\begin{ytableau}
		\none \\
		\none & \none & \\
		\none & \none & \\
		\none & \none & \\
		\none & *(lgray) & \\
		\none & *(lgray) & & \\
		\none & *(lgray) & & & & \none
	\end{ytableau}
	$\to$
	\begin{ytableau}
		\none \\
		\none & \\	
		\none & \\
		\none & \\
		\none & *(lgray) & \\
		\none & *(lgray) & & \\
		\none & *(lgray) & & &
	\end{ytableau}
\caption{A move in Bulgarian solitaire from  $\lambda=(7,3,2, 0, 0, \dots)\in\mathcal{P}(12)$: The bottom layer is picked to form a new pile with three cards and the cards are then left-shifted.
\label{fig:ordinary_move}}
\end{figure}


Olson \cite{Olson2016} recently introduced a generalization of Bulgarian solitaire in which the number of cards picked from a pile of size $h$ is given by  $\sigma(h)$, where $\sigma: \mathbb{Z}_{+} \rightarrow \mathbb{N}$ can be any function such that $\sigma(h)\le h$ for all $h\in \mathbb{Z}_{+}$. The ordinary Bulgarian solitaire is obtained for the constant function $\sigma(h)=1$. Olson studied cycle lengths, proving a general upper bound on cycle lengths for any specification of $\sigma$. 

The key results known to hold for ordinary Bulgarian solitaire do not generalize to Olson's $\sigma$-Bulgarian solitaire. For instance, it is well-known that ordinary Bulgarian solitaire has the property that if a stable (i.e., fixpoint) configuration exists it is unique \cite{Brandt}. This uniqueness property of stable configurations does not generally hold for the $\sigma$-Bulgarian solitaire; a simple counter-example is obtained by defining $\sigma(1)=1$, $\sigma(2)=1$, and $\sigma(3)=3$, in which case both (2,1) and (3) are stable configurations of three cards. To avoid such pathological cases we impose some additional conditions on $\sigma$:

\begin{definition}\label{def:well-behaved}
A $\sigma$-Bulgarian solitaire is said to be \emph{well-behaved} if the following three conditions on $\sigma$ are satisfied:
\begin{enumerate}
\item $\sigma(1)=1$,
\item $\sigma(h)$ is a non-decreasing function,
\item $\bar\sigma(h):=h-\sigma(h)$ is a non-decreasing function.
\end{enumerate}
\end{definition}

The first condition in the definition says that from a pile with just a single card, you pick that card. The second condition says that you never pick fewer cards from a larger pile than from a smaller pile. The  third condition says that the number of unpicked cards are never fewer in the larger pile than in a smaller pile. 

The aim of this paper is to show how a number of properties that are well-known to hold for ordinary Bulgarian solitaire generalize to well-behaved $\sigma$-Bulgarian solitaire. A simple example is the dominance property: If a configuration $\lambda$ is dominated by another configuration $\kappa$, in the sense that $\lambda_i \le \kappa_i$ holds for all $i$, then this dominance relation is preserved as the solitaire is played in parallel from the two configurations. Let $\lambda \le \kappa$ denote the dominance relation and let $\lambda^\text{new}$ and $\kappa^\text{new}$ denote the configurations obtained from playing one move of $\sigma$-Bulgarian solitaire from configurations $\lambda$ and $\kappa$, respectively.
 
\begin{theorem}\label{thm:dominance}
The implication  $\lambda \le \kappa \Rightarrow \lambda^\text{new} \le \kappa^\text{new}$ holds in $\sigma$-Bulgarian solitaire if both $\sigma$ and $\bar\sigma$ are non-decreasing functions. In particular, the implication holds for well-behaved $\sigma$.
\end{theorem}
\begin{proof}
If $\bar\sigma$ is non-decreasing, what remains of the old piles of $\lambda$ will be dominated by what remains of the old piles of $\kappa$.  If $\sigma$ is non-decreasing, the new pile formed from $\lambda$ will be dominated by the new pile formed from $\kappa$. This pilewise dominance clearly remains when the piles in each configuration are sorted by size. By Definition~\ref{def:well-behaved}, all well-behaved $\sigma$ satisfy that both $\sigma$ and $\bar\sigma$ are non-decreasing functions.
\end{proof}

We now outline the other properties to be examined in this paper. In section~\ref{well-behaved-interpretation} we show that moves of any well-behaved $\sigma$-solitaire have a geometric interpretation in terms of picking certain layers of cards. In section~\ref{sec:stable} we demonstrate that stable configurations of any well-behaved $\sigma$-solitaire are unique for any $n$ for which a stable configuration exists. In section~\ref{sec:convex} we characterize stable configurations of well-behaved $\sigma$-solitaires as convex (i.e., $\lambda_i-\lambda_{i+1} \ge \lambda_{i+1} - \lambda_{i+2}$ for all $i\ge 1$). In sections~\ref{sec:limit-concept} and \ref{sec:stable-limit} we define limit shapes of stable configurations for an infinite sequence of well-behaved $\sigma_n$-solitaires and show that any convex shape can be obtained as a limit shape of such a sequence. 

In section~\ref{sec:conjecture} we define the surplus and deficit of a configuration with respect to a given stable configuration and show that the total surplus and total deficit can decrease but never increase as a well-behaved $\sigma$-solitaire is played. This property suggests that as a well-behaved $\sigma$-solitaire is played the configurations tend to converge toward stable configurations. We conjecture that recurrent configurations are so close to stable configurations that they have the same limit shape. In sections~\ref{sec:q-stable} and \ref{sec:q-recurrent} we prove this conjecture in the special case when the $\sigma_n$-solitaires are given by $\sigma_n(h)=\lceil q_n h\rceil$ for $q_n\in (0,1]$. The limit shapes of stable and recurrent configurations are then triangular in case $q_n^2 n\rightarrow 0$, and exponential in case $q_n^2 n\rightarrow \infty$.

\section{A geometric interpretation of well-behaved $\sigma$-Bulgarian solitaires}\label{well-behaved-interpretation}
The moves of a well-behaved $\sigma$-Bulgarian solitaire has an intuitive geometric interpretation. First note that $\sigma$ is well-behaved if and only if it satisfies the boundary condition $\sigma(1)=1$ and the condition that for all pile sizes $h>0$ the difference $\Delta\sigma(h) := \sigma(h)-\sigma(h-1)$ equals either 1 or 0. We can then record the values of $h$ for which $\Delta\sigma(h)=1$ as a (finite or infinite) sequence $H_1=1, H_2, H_3, \dots$.  For any index $i$ for which $H_i$ is defined it follows follows from the definition that $H_i = \min\{h : \sigma(h)=i\}$. A move of the $\sigma$-Bulgarian solitaire can now be described as a move on the Young diagram of the configuration in which the cards in layers number $H_1=1, H_2, H_3, \dots$ (counted from the bottom) are removed to form a new column; sorting is then achieved by left-shifting the cards in every layer. See Figure~\ref{fig:wellbehaved_move} for an example. Note that ordinary Bulgarian solitaire is the special case in which only row number $H_1=1$ is removed to form a new column.

\begin{figure}[ht]
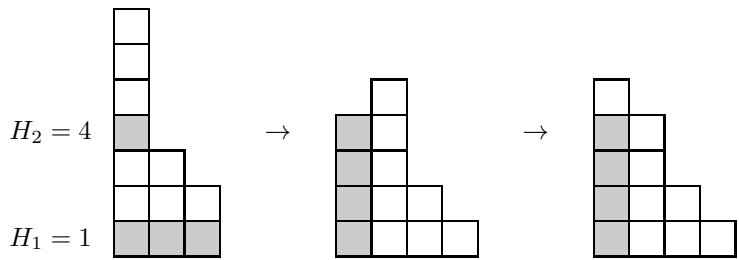
	
	\centering
	\begin{ytableau}
		\none \\
		\none \\
		\none \\
		\none[H_2=4] \\	
		\none \\
		\none \\
		\none[H_1=1]
	\end{ytableau}	
	\begin{ytableau}
		\none & \\
		\none & \\
		\none & \\
		\none & *(lgray) \\	
		\none & & \\
		\none & & & \\
		\none & *(lgray) & *(lgray) & *(lgray) & \none
	\end{ytableau}
	$\to$
	\begin{ytableau}
		\none \\
		\none \\
		\none & \none & \\
		\none & *(lgray) & \\
		\none & *(lgray) & \\
		\none & *(lgray) & & \\
		\none & *(lgray) & & & & \none
	\end{ytableau}
	$\to$
	\begin{ytableau}
		\none \\
		\none \\	
		\none & \\
		\none & *(lgray) & \\
		\none & *(lgray) & \\
		\none & *(lgray) & & \\
		\none & *(lgray) & & &
	\end{ytableau}
\caption{A move from the partition $\lambda=(7,3,2, 0, 0, \dots)\in\mathcal{P}(12)$ in a well-behaved $\sigma$-Bulgarian solitaire in which layers number $H_1=1$ and $H_2=4$ are picked to form a new pile with four cards.
\label{fig:wellbehaved_move}}
\end{figure}

\begin{observation}\label{obs:interpretation}
In the geometric interpretation of a well-behaved $\sigma$-Bulgarian solitaire, $\sigma(h)$ equals the number of picked layers up to layer $h$.
\end{observation}

In a move of a well-behaved $\sigma$-Bulgarian solitaire, the layers at levels  $H_1=1, H_2, H_3, \dots$ are picked and any other layer, say at level $h$, drops $\sigma(h)$ levels down (with one of the picked card inserted at the left end).  A layer may continue to drop down for several moves but must eventually reach one of the levels  $H_1=1, H_2, H_3, \dots$ and there be picked. Thus, any card currently at level $h$ will eventually be picked at some level $P(h)$ that can be calculated recursively by 
\[
P(h) =
\begin{cases}
      h & \text{ if } h\in \{H_1=1, H_2, H_3, \ldots \}, \\
      P(h-\sigma(h)) & \text{ otherwise}.
\end{cases}
\]
The following observation then follows immediately by induction.

\begin{observation}\label{obs:induction}
In the geometric interpretation of a well-behaved $\sigma$-Bulgarian solitaire, it holds for any level $h$ that the set
\[
\{P(h), P(h-1), P(h-2), \dots, P(h-\sigma(h)+1)\}
\]
is some permutation of the set $\{H_1=1, H_2,\dots, H_{\sigma(h)}\}$.
\end{observation}

\section{Uniqueness of stable configurations}\label{sec:stable}
Recall that a configuration is said to be stable with respect to $\sigma$-Bulgarian solitaire if a move in the solitaire leaves the configuration invariant. 

\begin{lemma}\label{lm:stable}
On the condition that $\bar\sigma$ is a non-decreasing function, $\lambda$ is a stable configuration with respect to $\sigma$-Bulgarian solitaire if and only if $\lambda_{i+1}= \bar\sigma(\lambda_i)$ for all $i\ge 1$.
\end{lemma}
\begin{proof}
A move of the solitaire decreases the size of any pile from $\lambda_i$ to $\bar\sigma(\lambda_i)$ and then creates a new pile such that the sum of all pile sizes stays constant at $n$, the total number of cards. Because $\bar\sigma$ is assumed to be a non-decreasing function, the decreased piles will still satisfy $\bar\sigma(\lambda_i)\ge\bar\sigma(\lambda_{i+1})$ for all $i\ge 1$, that is, they will not need to be reordered. Therefore $\lambda$ is a stable configuration if and only if $\bar\sigma(\lambda_i) = \lambda_{i+1}$ for all $i\ge 1$, as the new pile will then necessarily have size $\lambda_1$.
\end{proof}

\begin{theorem}\label{thm:q-unique}
If $\bar\sigma$ is a non-decreasing function, there is at most one stable configuration with respect to $\sigma$-Bulgarian solitaire for a given number $n$ of cards.
\end{theorem}
\begin{proof}
When $\bar\sigma$ is a non-decreasing function it follows from Lemma~\ref{lm:stable} that each stable configuration is completely determined by the choice of $\lambda_1$, the size of the largest pile. Let $\lambda$ be a stable configuration with $n$ cards and consider another stable configuration $\lambda'$ with $\lambda_1>\lambda'_1$. It follows immediately by induction that $\lambda_i\ge \lambda'_i$ for all $i\ge 1$, and consequently that the total number of cards in these two configurations are different.
\end{proof}

We can also bound the difference in the total number of cards between consecutive stable configurations in a well-behaved $\sigma$-solitaire. 

\begin{corollary}\label{cor:sigma-nextstable}
Assume that $\sigma(1)=1$ and that both $\sigma$ and $\bar\sigma$ are non-decreasing functions and let $\lambda$ and $\lambda'$ be the stable configurations (with respect to $\sigma$-Bulgarian solitaire) determined by first piles of size $\lambda_1$ and $\lambda'_1 = \lambda_1 + 1$, respectively. Then the difference in the total number of cards between $\lambda'$ and $\lambda$ is at most $\ell(\lambda)+1$. 
\end{corollary}
\begin{proof}
As we noted in the introduction, the assumption that both $\sigma$ and $\bar\sigma$ are non-decreasing functions implies that for any pile size $h$ we have that $\sigma(h+1)-\sigma(h)$ equals either 1 or 0. Starting from the relation $\lambda'_1 = \lambda_1 + 1$, it follows immediately by induction that as long as $\sigma(\lambda_i+1)-\sigma(\lambda_i) = 0$ we will also have $\lambda'_{i+1} = \lambda_{i+1} + 1$. The first time we instead have $\sigma(\lambda_i+1)-\sigma(\lambda_i) = 1$, we will obtain $\lambda'_{i+1} = \lambda_{i+1}$, and from that point on the pile sizes will be identical in the two configurations. Thus, the difference in the total number of cards is equal to the number of piles that differed in size, which is at most the number of piles in the larger configuration $\lambda'$. Because each of its piles is at most one larger than the corresponding piles in the smaller configuration $\lambda$, it can have at most one pile more. Hence, the difference in the total number of cards is bounded by $\ell(\lambda)+1$. 
\end{proof}

\section{Characterization of stable configurations}\label{sec:convex}
We shall now characterize what stable configurations of well-behaved $\sigma$-solitaires look like. Define a configuration $\lambda$ as \emph{convex} if $\lambda_i - \lambda_{i+1} \ge \lambda_{i+1} - \lambda_{i+2}\ge 0$ for all $i\ge 1$.

\begin{lemma}\label{lm:stable-finite-shape}
A configuration $\lambda$ is convex if and only if it is a stable configuration of a well-behaved $\sigma$-Bulgarian solitaire.
\end{lemma}
\begin{proof}
First assume that $\lambda$ is a stable configuration of a well-behaved $\sigma$-Bulgarian solitaire. Then Lemma~\ref{lm:stable} says that $\lambda_i - \lambda_{i+1} = \sigma(\lambda_i)\ge 0$ for all $i\ge 1$. As a well-behaved $\sigma$ is non-decreasing, this inequality implies that $\lambda$ is convex.

To prove the converse, assume that $\lambda$ is a convex configuration with $\ell$ nonzero piles. Then for each $i\ge 1$ we can choose a subset of $(\lambda_i - \lambda_{i+1})-(\lambda_{i+1}-\lambda_{i+2})$ layers in the interval of layers $(\lambda_{i+1},\lambda_i]$. Note that this means all layers in the interval $(0,\lambda_\ell]$ are chosen, in particular layer 1. A well-behaved $\sigma$-Bulgarian solitaire is therefore defined by picking the chosen layers in each move. Moreover, for all $i\ge 1$ the corresponding $\sigma$ will satisfy $\sigma(\lambda_i)=\lambda_i - \lambda_{i+1}$ as the latter expression equals the number of picked layers up to layer $\lambda_i$ (see Observation~\ref{obs:interpretation}). Thus, $\lambda$ is a stable configuration of this  well-behaved $\sigma$-Bulgarian solitaire.
\end{proof}

\section{The concept of limit shapes of stable and recurrent configurations}\label{sec:limit-concept}
We shall consider a sequence of well-behaved $\sigma_n$, for $n=1,2,\dots$.  Let $f_{\sigma_n} :\mathcal{P}(n) \rightarrow \mathcal{P}(n)$ denote the map on integer partitions of $n$ defined by a move of the $\sigma_n$-Bulgarian solitaire. For any $n$, the $\sigma_n$-Bulgarian solitaire can be regarded as the deterministic process on $\mathcal{P}(n)$ defined by the iteration of the map $f_{\sigma_n}$. 

Given a process on $\mathcal{P}(n)$ defined by the iteration of any map $f_n:\mathcal{P}(n)\rightarrow\mathcal{P}(n)$, we say that a configuration $\lambda\in\mathcal{P}(n)$ is \emph{stable} with respect to this process in case $f(\lambda)=\lambda$, and that the configuration is \emph{recurrent} if there exists a positive integer $k$ such that $f_n^k(\lambda)=\lambda$. Thus, the stable configurations constitute a subset of the recurrent configurations. Note that the set of all configurations, $\mathcal{P}(n)$, is finite. Regardless of choice of starting configuration, the process must therefore inevitably enter the set of recurrent configurations after a finite number of moves. We shall now define what we mean by limit shapes of stable or recurrent configurations. 

\subsection{Downscaling of diagram-boundary functions}
\label{sec:downscaling}
For any partition $\lambda$, define its \textit{diagram-boundary function} as the nonnegative, weakly decreasing and piecewise constant function $\partial\lambda:\Rnn\rightarrow\mathbb{R}$ given by
\[
\partial\lambda(x)=\lambda_{\lfloor x \rfloor+1}.
\]
To illustrate, Figure~\ref{fig:yd_example} depicts the function graph $y=\partial\lambda(x)$ for the partition $\lambda=(4,4,2,1,1, 0, 0, \dots)$. 

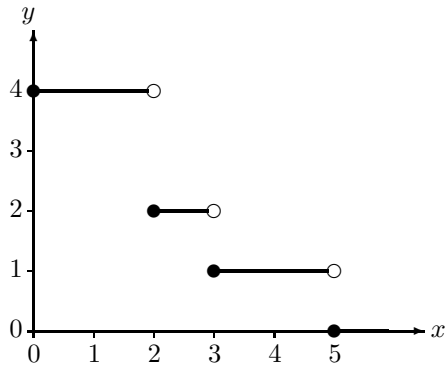
\begin{figure}[ht]
\setlength{\unitlength}{0.08cm}
\centering
\begin{picture}(75,60)
\put(4,5){\vector(0,1){50}}
\put(4,5){\vector(1,0){65}}
\put(3,0){$0$}
\put(4,4){\line(0,1){1}}
\put(13,0){$1$}
\put(14,4){\line(0,1){1}}
\put(23,0){$2$}
\put(24,4){\line(0,1){1}}
\put(33,0){$3$}
\put(34,4){\line(0,1){1}}
\put(43,0){$4$}
\put(44,4){\line(0,1){1}}
\put(53,0){$5$}

\put(0,4){$0$}
\put(3,5){\line(1,0){1}}
\put(0,14){$1$}
\put(3,15){\line(1,0){1}}
\put(0,24){$2$}
\put(3,25){\line(1,0){1}}
\put(0,34){$3$}
\put(3,35){\line(1,0){1}}
\put(0,44){$4$}

\linethickness{1.3pt}
\put(4,45){\circle*{2}}
\put(24,45){\circle{2}}
\put(4,45){\line(1,0){19}}

\put(24,25){\circle*{2}}
\put(34,25){\circle{2}}
\put(24,25){\line(1,0){9}}

\put(34,15){\circle*{2}}
\put(34,15){\line(1,0){19}}
\put(54,15){\circle{2}}

\put(54,5){\circle*{2}}
\put(54,5){\line(1,0){9}}

\put(70,4){$x$}
\put(2,57){$y$}
\end{picture}

\caption{Function graph $y=\partial\lambda(x)$ for the partition $\lambda=(4,4,2,1,1)\in\mathcal{P}(12)$.\label{fig:yd_example}}
\end{figure}

To achieve limiting behavior of such function graphs as $n$ grows we need to rescale diagrams depending on the value of $n$. Following \cite{Eriksson2012575} and \cite{VershikStatMech} we apply a scaling factor ${a_n>0}$ such that all row lengths are multiplied by $1/a_n$ and all column heights are multiplied by $a_n/n$, yielding a constant area of 1. To avoid having to specify the scaling factor we shall consistently make the choice $a_n = n/\lambda_1$, such that the height of the diagram is also scaled to 1. Thus, given a partition $\lambda$ we define the \emph{downscaled} diagram-boundary function of $\lambda$ as the positive, real-valued, weakly decreasing and piecewise constant function $\rescaled{\lambda}:\Rnn\rightarrow\Rnn$ given by
\begin{equation}
\label{eq:def_rescaled}
\rescaled{\lambda}(x)=\dfrac{1}{\lambda_1}\partial\lambda(xn/\lambda_1)=\frac{1}{\lambda_1}\lambda_{\lfloor xn/\lambda_1 \rfloor+1}.
\end{equation}

\subsection{Limit shapes of recurrent configurations}
Given an infinite family of maps on configurations, $\{ f_n : \mathcal{P}(n)\rightarrow\mathcal{P}(n) \}$ for $n=1,2,3,\dots$, we say that $\phi:\Rnn\rightarrow\Rnn$ is a \emph{limit shape of recurrent configurations} of the corresponding family of processes if the downscaled diagrams of recurrent configurations with respect to $f_1, f_2, f_3, \dots$ converge to $\phi$ in the following sense:
\begin{equation}
\label{eq:deterministic_limitshape_def}
\partial{\lambda^{(n)}}(x) \to \phi(x) \text{ as }n\to\infty
\end{equation}
for all $x>0$, where $\lambda^{(n)}$ is any recurrent configuration of $f_n$.

\subsection{Limit shapes of stable configurations of well-behaved $\sigma_n$-solitaires}
Consider a sequence of well-behaved $\sigma_n$ for $n=1,2,\dots$. For each value of $n$, consider $\sigma_n$-Bulgarian solitaire and let 
\[
n^{\ast}=n^{\ast}(n):= \max\{n'\le n : \text{ there exists a stable configuration of $n'$ cards\}}.
\]
The number $n^{\ast}$ is well-defined as there always exists a stable configuration on a single card. Define $\lambda^{(n^\ast)}$ as the stable configuration on $n^{\ast}$ cards; this is well-defined according to Theorem~\ref{thm:q-unique}. According to Corollary~\ref{cor:sigma-nextstable}, the choice $\lambda^{(n^\ast(n))}$ as the stable configuration corresponding to $n$ will be the same for at most $\ell(\lambda^{(n^\ast)(n)})+1$ values of $n$. Thus, $n^{\ast}(n)$ tends to infinity as $n$ tends to infinity. In line with the limit shape concept defined in equation (\ref{eq:deterministic_limitshape_def}), we define a \emph{limit shape of stable configurations} for the sequence of well-behaved $\sigma_n$, $n=1,2,\dots$, as a function $\phi:\Rnn\rightarrow\Rnn$ such that 
\begin{equation}
\label{eq:stable_limitshape_def}
\partial{\lambda^{(n^{\ast}(n))}}(x) \to \phi(x) \text{ as }n\to\infty
\end{equation}
for all $x>0$.

\section{Characterization of limit shapes of stable configurations of well-behaved $\sigma_n$-solitaires}\label{sec:stable-limit}
It is well-known that the Bulgarian solitaire has a stable configuration if only if the total number of cards in the deck is a triangular number, $n=1+2+\ldots+k$ for some positive integer $k$, in which case the unique stable configuration has one pile of each integer size from $k$ down to $1$ \cite{Brandt}. Thus, the Young diagrams of stable configurations are staircase shaped and hence the limit shape, as $n$ tends to infinity, is a triangle. By the convention that the height and area are scaled to unity, the limit shape will be a triangle of height 1 and width 2.

When generalizing from ordinary Bulgarian solitaire to well-behaved $\sigma_n$-Bulgarian solitaire, the limit shapes that arise will not necessarily be triangular. Indeed, in Theorem~\ref{thm:stable-limit-convex} we prove that any convex shape 
(with some properties)
can be obtained as the limit shape of a suitably chosen family of well-behaved $\sigma_n$, for $n=1,2,\dots$. 
First we need a lemma.

\begin{lemma}\label{lm:convexdiff}	
	Let $f_n\,:\,(0,\infty)\rightarrow\mathbb{R}$ be convex functions	
	for $n=1,2,\dotsc$, and suppose there is a function $f\,:\,(0,\infty)\rightarrow\mathbb{R}$	
	such that $\lim_{n\rightarrow\infty}f_n(x)=f(x)$ for any $x>0$.	
	Then we also have pointwise convergence of right derivatives:	
	$\lim_{n\rightarrow\infty}(f_n)_R'(x)=f_R'(x)$ for any $x>0$.
\end{lemma}

\begin{proof}
	Clearly, $f$ is convex. Therefore the left derivative $f_L'$ and the right derivative $f_R'$ exists in all points in $(0,\infty)$.
	By the definition of right derivative, for any $\varepsilon > 0$ there exists an $h>0$ such that
	\[
	\frac{f(x+h)-f(x)}{h} < f_R'(x)+\varepsilon
	\]
	for all $x>0$. Since $\lim_{n\to\infty}f_n(x) = f(x)$,
	there exists an $N\in\mathbb{N}$ such that
	\[
	\frac{f_n(x+h)-f_n(x)}{h} < f_R'(x)+\varepsilon \quad \text{for all }n>N.
	\]
	By the convexity of $f_n$, we have $\frac{f_n(x+h)-f_n(x)}{h} \ge f'_n(x)$.
	Therefore,
	\[
	f'_n(x) 
	< f_R'(x)+\varepsilon  \quad \text{for all }n>N.
	\] 
	An analogous reasoning for the left-derivative gives $f'_n(x) > f_L'(x)-\varepsilon$.	
	Thus, $f_L'(x)-\varepsilon < f'_n(x) < f_R'(x)+\varepsilon$ for all $n>N$.
	As a consequence we must have $\lim_{n\to\infty}f'_n(x) = f'(x)$, in particular
	$\lim_{n\to\infty}(f_n)_R'(x) = f_R'(x)$ for any $x>0$.
\end{proof}

\begin{theorem}\label{thm:stable-limit-convex}
	Let $\phi:(0,\infty)\rightarrow\Rnn$ be a
	function and let	
	$a_1,a_2,\cdots\rightarrow\infty$ be any (positive) scaling factors	
	such that $a_n^2/n$ converges to some $c\ge0$ as $n\rightarrow\infty$.	
	Then the following are equivalent.	
	\begin{itemize}		
		\item[(a)]		
		There is a sequence of well-behaved $\sigma_n$, $n=1,2,\dots$, such that		
		$\phi$ is a stable-limit shape of $(\sigma_n)$ under the scaling		
		$(a_n)$.
		\item[(b)]		
		$\phi$ is convex with $\int_0^\infty\phi(x)\,dx\le1$,
		and if $c>0$		
		the right derivative $\phi_R'(x)$		
		is an integer multiple of $c$ for any $x>0$.		
	\end{itemize}
\end{theorem}

\begin{proof}	
	To prove that (a) implies (b), suppose $\phi$ is a stable-limit shape	
	of $(\sigma_n)$ under the scaling $(a_n)$. Let $\lambda^{(n^\ast)}$	
	denote the stable configuration in the $\sigma_n$-Bulgarian solitaire	
	with $n^\ast(n)$ cards.
	
	For each	
	$n$, define a piecewise linear function	
	$\phi_n\,:\,(0,\infty)\rightarrow\Rnn$ by letting	
	\[	
	\phi_n(x)=\frac{a_n}{n}	
	\bigl((1-t(x))\lambda^{(n^\ast)}_{\lfloor a_n x\rfloor+1}	
	+t(x)\lambda^{(n^\ast)}_{\lfloor a_n x\rfloor+2}\bigr),	
	\]	
	where $t(x):=a_n x-\lfloor a_n x\rfloor$. Since each $\phi_n$ is convex, so is $\phi$, and	
	by Fatou's lemma	
	\[	
	\int_0^\infty \phi(x)\,dx \le \liminf_{n\rightarrow\infty}	
	\int_0^\infty \phi_n(x)\,dx\le1.	
	\]
	
	Now suppose $c>0$.	
	By Lemma~\ref{lm:convexdiff}, $(\phi_n)_R'(x)\rightarrow\phi_R'(x)$	
	for any $x>0$, so	
	\[	
	\frac{(\phi_n)_R'(x)}{a_n^2/n}\rightarrow\phi_R'(x)/c	
	\]	
	as $n\rightarrow\infty$. But	
	\[	
	\frac{(\phi_n)_R'(x)}{a_n^2/n}=	
	(\lambda^{(n^\ast)}_{\lfloor a_n x\rfloor+2}	
	-\lambda^{(n^\ast)}_{\lfloor a_n x\rfloor+1}),	
	\]	
	which is an integer, so it follows that $\phi_R'(x)/c$ is an integer.
	
	For the other direction, suppose (b) holds true and that $c=0$.	
	Since $\phi$ is convex, it has a right derivate $\phi'_R$.	
	Let $s_1,s_2,\dotsc$ be a sequence of positive real numbers such that	
	$s_n\rightarrow\infty$ but $s_na_n^2/n\rightarrow0$ as	
	$n\rightarrow\infty$, and such that	
	$s_na_n$ is an integer for any $n$.
		
	Define a partition $\lambda^{(n)}$ by letting	
	\[	
	\lambda^{(n)}_k =	
	\sum_{i=k+1}^{s_na_n}	
	\lfloor-\frac{n}{a_n^2}\phi'_R(i/a_n)\rfloor	
	\]	
	for $k=1,2,\dotsc$.	
	By the convexity of $\phi$, we have	
	\begin{align*}	
	\lambda^{(n)}_k &\le -\frac{n}{a_n}	
	\int_{k/a_n}^\infty\phi'_R(x)\,dx	
	=\frac{n}{a_n}\,\phi(k/a_n), \\	
	\lambda^{(n)}_k &\ge -\frac{n}{a_n}	
	\int_{(k+1)/a_n}^{s_n}\phi'_R(x)\,dx	
	-s_na_n	
	=\frac{n}{a_n}\bigl(\phi((k+1)/a_n)-\phi(s_n)\bigr)-s_na_n.	
	\end{align*}	
	From the first of these inequalities it follows that	
	\[	
	\sum_{k=1}^\infty\lambda^{(n)}_k	
	\le a_n\int_0^\infty\frac{n}{a_n}\,\phi(x)\,dx \le n.	
	\]	
	Now, let $\mu^{(n)}_k=\lambda^{(n)}_k$ for $k=2,3,\dotsc$ but	
	choose $\mu^{(n)}_1$ so that $\mu^{(n)}_1+\mu^{(n)}_2+\dotsb=n$.	
	Clearly, $\mu^{(n)}$ is monotonic convex, so by	
	Lemma~\ref{lm:stable-finite-shape} it is a stable configuration of	
	a $\sigma_n$-Bulgarian solitaire for some well-behaved rule $\sigma_n$.	
	Since $\phi$ is continuous, for any $x>0$ we have	
	\[	
	\frac{a_n}{n}\lambda_{\lceil a_n x\rceil}	
	\le\phi(\lceil a_n x\rceil/a_n)	
	\rightarrow\phi(x)	
	\]	
	and, since $s_n\rightarrow\infty$ and $s_n a_n^2/n\rightarrow 0$ as $n\rightarrow\infty$, for any $x>0$,	
	\[	
	\frac{a_n}{n}\lambda_{\lceil a_n x\rceil}	
	\ge	
	\phi(\lceil a_n x+1\rceil/a_n)-\phi(s_n)-\frac{s_na_n^2}{n}	
	\rightarrow \phi(x),
	\]
	in other words, $\frac{a_n}{n}\lambda_{\lceil a_n x\rceil} \to \phi(x)$ for any $x>0$,
	establishing that $\phi$ is the desired stable-limit shape by the definition \eqref{eq:stable_limitshape_def}.
	
	Now suppose (b) holds true and $c>0$.	
	Define a partition $\lambda^{(n)}$ by letting	
	\[	
	\lambda^{(n)}_k = -\frac1c\sum_{i=k+1}^\infty\phi'_R(i/\sqrt{cn})	
	\]
	for $k=1,2,\dotsc$.	
	Since $\phi$ is convex, we have	
	\begin{align*}	
	\lambda^{(n)}_k &\le -\frac{\sqrt{cn}}{c}	
	\int_{k/\sqrt{cn}}^\infty\phi'_R(x)\,dx	
	=\sqrt{\frac{n}{c}}\,\phi(k/\sqrt{cn}), \\	
	\lambda^{(n)}_k &\ge -\frac{\sqrt{cn}}{c}	
	\int_{(k+1)/\sqrt{cn}}^\infty\phi'_R(x)\,dx	
	=\sqrt{\frac{n}{c}}\,\phi((k+1)/\sqrt{cn}).	
	\end{align*}	
	From the first of these inequalities it follows that	
	\[	
	\sum_{k=1}^\infty\lambda^{(n)}_k	
	\le\sqrt{cn}\int_0^\infty\!\!\sqrt{\frac{n}{c}}\,\phi(x)\,dx \le n.	
	\]
		
	Now, let $\mu^{(n)}_k=\lambda^{(n)}_k$ for $k=2,3,\dotsc$ but	
	choose $\mu^{(n)}_1$ so that $\mu^{(n)}_1+\mu^{(n)}_2+\dotsb=n$.	
	Clearly, $\mu^{(n)}$ is monotonic convex, so by	
	Lemma~\ref{lm:stable-finite-shape} it is a stable configuration of	
	a $\sigma_n$-Bulgarian solitaire for some well-behaved rule $\sigma_n$.	
	Finally, since $a_n^2/n\rightarrow c$ as $n\rightarrow\infty$, and	
	since $\phi$ is continuous, for any $x>0$ we have	
	\[	
	\phi(x)\leftarrow\frac{a_n}{\sqrt{cn}}\phi(\lceil a_n x+1\rceil/\sqrt{cn})	
	\le\frac{a_n}{n}\lambda_{\lceil a_n x\rceil}	
	\le\frac{a_n}{\sqrt{cn}}\phi(\lceil a_n x\rceil/\sqrt{cn})	
	\rightarrow\phi(x),
	\]	
	and hence $\phi$ is the desired limit shape.
\end{proof}

Note that for a finite $n$ a Young diagram will have unit area under our conventional scaling
with scaling factors $(a_n)$.
The reason why we have $\int_0^\infty\phi(x)\,dx\le 1$ in Theorem~\ref{thm:stable-limit-convex} 
is that the largest pile (or a few of the largest piles) may be arbitrarily large without
affecting the limit shape $\phi$. Our limit shape definitions
\eqref{eq:deterministic_limitshape_def} and \eqref{eq:stable_limitshape_def} do not include $x=0$
as to allow for $\lim_{x\to 0+}\phi(x)$ to be infinite.


\section{A conjecture on limit shapes of recurrent configurations of well-behaved $\sigma_n$-solitaires}\label{sec:conjecture}

When a stable configuration exists, ordinary Bulgarian solitaire eventually reaches it. This does not hold in general for well-behaved $\sigma$-Bulgarian solitaire. One counter-example is given by the well-behaved $\sigma$-Bulgarian solitaire defined by $\sigma(h)=\lceil 3h/10\rceil$ on $n=11$ cards, as this game has both a stable configuration $(5,3,2,1)$ and a non-trivial cycle 
\[
(6,2,2,1) \mapsto (5,4,1,1) \mapsto (6,3,2) \mapsto (4,4,2,1) \mapsto (6,2,2,1).
\]
However, it is worth noting that the pile sizes in these recurrent configurations never deviated by more than one card from the corresponding pile sizes in the stable configuration. This is akin to the ordinary Bulgarian solitaire in the case when $n$ is not a triangular number so that no stable configuration exists; in that case the solitaire will eventually reach a cycle of recurrent configurations, and these are close to staircase shaped in the sense that they can all be constructed by starting with some staircase configuration $(k,k-1,\ldots,1)$ and adding at most one card to each pile, and possibly adding one more pile of size 1 \cite{AkinDavis1985,Bentz1987, Etienne1991, Griggs1998}. As $n$ grows to infinity and the diagram is rescaled such that its height and area are both equal to 1, the deviation of recurrent configurations from the perfect staircase tends to zero. Thus, ordinary Bulgarian solitaire has a limit shape, namely the triangle of height 1 and width 2. 

We believe that it is generally true that recurrent configurations must be sufficiently close to a stable configuration for a limit shape of stable configurations to also be a limit shape of recurrent configurations.

\begin{conjecture}\label{conj:stable-limit-is-limit}
If $\phi$ is a limit shape of the stable configurations of a sequence of well-behaved $\sigma_n$, then $\phi$ is also a limit shape of the recurrent configurations. 
\end{conjecture}

We leave this conjecture as an open problem. A first step toward its proof is that a configuration's total deviation from a stable configuration will often decrease but never increase during play, as we show below.

\subsection{Deviations, surplus, and deficit, with reference to a stable configuration}\label{sec:surplus}
Represent configurations by infinite vectors of pile sizes (a finite number of non-zero piles and an infinite tail of zeros). For a given $\sigma$-Bulgarian solitaire, fix some stable configuration $\lambda^{\ast}$ to be used as a reference and let $n^{\ast}$ denote the number of cards of $\lambda^{\ast}$. 

For any configuration $\lambda$ of $n$ cards we can calculate the component-wise difference to the reference configuration $\lambda^{\ast}$ to obtain an infinite \emph{deviation} vector $d(\lambda):=\lambda-\lambda^{\ast}$. The sum of the elements of the deviation vector must equal the difference in the number of cards in the two configurations: $d_1 + d_2 + \dotsb = n-n^{\ast}$, suppressing the dependence on $\lambda$ to avoid cumbersome notation.

The deviation vector can be decomposed as $d(\lambda)=d^+(\lambda) - d^-(\lambda)$ where the \emph{surplus} vector $d^+(\lambda)$ is given by $d^+_i = \max(d_i,0)$ for all $i\ge 1$, and the \emph{deficit} vector $d^-(\lambda)$ is given by  $d^-_i = \max(-d_i,0)$ for all $i\ge 1$. Denote the total surplus and total deficit in by $d^+_\text{tot}(\lambda)$ and $d^-_\text{tot}(\lambda)$, respectively. Then, obviously, 
\[
d^+_\text{tot}(\lambda) - d^-_\text{tot}(\lambda) = n-n^{\ast}.
\]
Thus, the difference between total surplus and total deficit is invariant under play.

\subsection{The total surplus never decreases in well-behaved $\sigma$-Bulgarian solitaires}
It is well-known and easy to see that ordinary Bulgarian solitaire has the property that a move can never increase the total surplus:  $d^+_\text{tot}(\lambda) \ge d^+_\text{tot}(\lambda^\text{new})$ will always hold. This property generalizes to well-behaved $\sigma$-Bulgarian solitaire. To prove this we shall extend the geometric interpretation of the solitaire by marking some cards as \emph{plus-cards} or \emph{minus-cards} with special properties. 

Given a reference stable configuration $\lambda^{\ast}$ of $n^\ast$ cards, any configuration $\lambda$ of $n$ cards can be transformed to a marked version of $\lambda$ in the following two steps: 
\begin{enumerate}
\item \emph{Marking plus-cards}: For any pile index $i$ such that $\lambda_i>\lambda^{\ast}_i$ (i.e., such that the pile has a surplus $d^+_i=\lambda_i-\lambda^{\ast}_i>0$), mark the $d^+_i$ top cards of the pile as plus-cards. 
\item \emph{Creating minus-cards}: For any pile index $i$ such that $\lambda_i<\lambda^{\ast}_i$ (i.e., such that the pile has a deficit $d^-_i=\lambda^{\ast}_i-\lambda_i>0$), add an additional $d^-_i$ minus-cards to the top of the pile.
\end{enumerate}

\begin{observation}\label{obs:marked} 
In the marked version of $\lambda$ the plus-cards and unmarked cards together make up $\lambda$, while the minus-cards and unmarked cards together make up $\lambda^\ast$. It follows that the number of plus-cards equals the total surplus and the number of minus-cards equals the total deficit.
\end{observation}
 
Recall the geometric interpretation of well-behaved $\sigma$-Bulgarian solitaire described in section~\ref{well-behaved-interpretation}: Remove the cards in layers number $H_1=1, H_2, H_3, \dots$ (counted from the bottom) and let these cards form a new first pile, then sort the configuration by left-shifting the cards in every layer. The rules of play on marked configurations follow this interpretation, with the following special rules for marked cards.

\emph{Cancellation within the new pile}: When a new first pile is formed, all minus-cards and plus-cards in this pile float to the top of the pile. If there are both minus-cards and plus-cards in the pile then cancellation occurs: Repeatedly replace one minus-card and  one plus-card by a single unmarked card until either there are no more minus-cards or there are no more plus-cards in the pile.

\emph{Cancellation within a layer}: Left-shifting layers clearly corresponds to minus-cards in the first pile swapping places with any cards to the right in the same layer. If at any point a minus-card that started in the first pile swaps places with a plus-card, replace the minus-card and the plus-card by a single unmarked card; then continue left-shifting as usual.

We shall refer to this process as the \emph{marked} $\sigma$-Bulgarian solitaire. 

\begin{lemma}\label{lm:marked}
If a well-behaved $\sigma$-Bulgarian solitaire from $\lambda$ is played in parallel with the marked $\sigma$-Bulgarian solitaire from the marked version of $\lambda$, Observation~\ref{obs:marked} will continue to hold.
\end{lemma}
\begin{proof}
First restrict attention to the movement of the minus-cards and unmarked cards. Together they make up the stable configuration $\lambda^\ast$ so, precisely because this is a stable configuration, these cards will still make up $\lambda^\ast$ after a move. Now restrict attention to the movement of plus-cards and unmarked cards. Together they make up $\lambda$ so clearly after one move they will make up $\lambda^\text{new}$. Finally, different types of cards interact only in the form of cancellations within the new pile or within a layer. In both forms of cancellation, the configuration made up of minus-cards and unmarked cards together is unchanged as in effect a minus-card is replaced by an unmarked card. The same goes for the configuration made up of plus-cards and unmarked cards together.
\end{proof}

From Lemma~\ref{lm:marked} the desired result follows immediately.

\begin{theorem}\label{thm:total-surplus}
Total surplus and total deficit may decrease but never increase during play of a well-behaved $\sigma$-solitaire.
\end{theorem}
\begin{proof}
From Lemma~\ref{lm:marked} and the rules of the marked $\sigma$-Bulgarian solitaire it follows that the total surplus equals the number of plus-cards and the total deficit equals the number of minus-cards, which decrease when cancellations occur but never increase.
\end{proof}

In order to prove Conjecture~\ref{conj:stable-limit-is-limit} we now only need to prove that whenever the deviation from the closest stable configuration is large, a cancellation of a plus-card against a minus-card must eventually occur. 

A reason why it is likely that a cancellation will eventually occur is that when plus-cards start over from the first pile they obviously do so from a higher level than minus-cards. It will therefore generally take a greater number of moves for a plus-card than for a minus-card from when it starts over until it reaches the level of a picking layer where it starts over again. Thus minus-cards have shorter periods than plus-cards and should eventually catch up with them. For ordinary Bulgarian solitaire, which has only one picking layer, this argument is sufficient to prove that cancellations must occur until no more minus-cards remain. In general, however, well-behaved $\sigma$-solitaire has several picking layers and this makes it difficult to perform a rigorous study of periods. However, we still believe that it may be possible to overcome these technical difficulties to achieve a proof of Conjecture~\ref{conj:stable-limit-is-limit}.

\section{Limit shapes of stable configurations of $q_n$-proportion Bulgarian solitaire}\label{sec:q-stable}
In order to calculate explicit limit shapes we make a canonical choice of a well-behaved $\sigma$, namely $\sigma(h)=\lceil q h\rceil$ for $q\in (0,1]$. (It should be obvious that this function satisfies the conditions for being well-behaved, see Definition~\ref{def:well-behaved}.) In words, this form of $\sigma$ defines a solitaire in which from each pile we pick a number of cards given by the proportion $q$ of the pile size, rounded upward to the closest integer. We will refer to this solitaire as \emph{$q$-proportion Bulgarian solitaire}. Following the geometric interpretation of well-behaved $\sigma$-Bulgarian solitaire, the cards picked in $q$-proportion Bulgarian solitaire can be seen as layers number $H_{1}, H_{2}, H_{3}, \dots$, where for any $i>0$ we have $H_{i} = \min\{h : \lceil q_n h\rceil=i\}$. Thus, $q$-proportion Bulgarian solitaire is a well-behaved $\sigma$-Bulgarian solitaire in which the picked layers are approximately equidistant.

We may let the choice of $q$ depend on $n$, in which we write $q_n$. Note that for $q_n \le 1/n$ only one card is picked in any pile.  Thus by choosing $q_n\le 1/n$ we obtain ordinary Bulgarian solitaire. 

Thanks to Lemma~\ref{lm:stable} we can determine a unique stable configuration of a $q$-proportion solitaire by choosing the size of the largest part and then obtain the other parts by repeatedly applying the function $\bar\sigma(h) = h - \lceil q h\rceil$. This makes it easy to determine the limit shapes of stable configurations. Specifically, we identify three different regimes defined by the asymptotic behavior of $n q_n^2$. 

First, in case $n q_n^2$ tends to zero as $n$ tends to infinity, stable configurations have a triangular limit shape. This is a direct generalization of the limit shape result for the ordinary Bulgarian solitaire. 

The second regime is when $n q_n^2$ tends to infinity, in which case an exponential limit shape is obtained. The borderline regime when $n q_n^2$ tends to a constant $C$ yields an infinite family of limit shapes (parameterized by $C$), which interpolate between the triangular shape of the first regime and the exponential shape of the second regime.

A move of $q$-proportion Bulgarian solitaire involves rounding the number of picked cards in each pile to integers. The three regimes differ in how much impact this rounding has on the result. The following couple of lemmas estimate the impact of rounding.

\begin{lemma}\label{lm:new-pile}
After a move in $q$-proportion Bulgarian solitaire from a configuration with $m$ non-empty piles, the size of the new pile is $nq + r$ where $0\le r < m$. 
\end{lemma}
\begin{proof}
By the definition of $q$-proportion Bulgarian solitaire, the contribution to the new pile from any old non-empty pile $\lambda_i$ is $\lceil q \lambda_i \rceil$, which is bounded from below and from above by
\[
\lambda_i q \le \lceil q \lambda_i \rceil < \lambda_i q + 1.
\]
The lemma follows from summing over all $m$ non-empty piles.
\end{proof}

\begin{lemma}\label{lm:number-of-piles}
The number of moves in $q$-proportion Bulgarian solitaire until a pile of size $h$ disappears is at most 
\[ \frac{\ln(q h) +1}{q}.
\]
\end{lemma}
\begin{proof}
In each move a pile decreases at least by a factor of $(1-q)$. A pile starting at size $h$ will have gone down to size $1/q$ after at most $\ln(1/(q h))/\ln(1-q)$ moves, which (using the MacLaurin expansion of the denominator) in turn is bounded by $\ln(q h)/q$. From size $1/q$ and onward the pile will, due to rounding, lose exactly $1$ card per move for $1/q$ moves at which point the pile has disappeared. 
\end{proof}

We shall now derive the limit shape of stable configurations depending on the asymptotic behavior of $n q_n^2$.  See Figure~\ref{fig:shapes}.
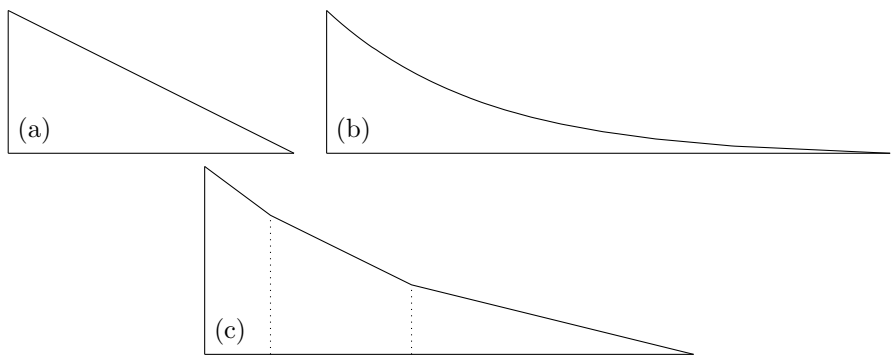
\begin{figure}[ht]
\centering
\begin{tabular}{cc}
\begin{tikzpicture}[scale=1.9]
\draw (0,0) -- (2,0);
\draw (0,0) -- (0,1);
\draw (0,1) -- (2,0);
\node [above right] at (0,0) {(a)};
\end{tikzpicture}
&
\begin{tikzpicture}[scale=1.9]
\draw (0,0) -- (3.9429,0);
\draw (0,0) -- (0,1);
\draw (0,1) -- 
(0.026396,0.97591) -- 
(0.052792,0.95181) -- 
(0.079188,0.92878) -- 
(0.10558,0.90589) -- 
(0.13198,0.88407) -- 
(0.15838,0.86225) -- 
(0.18477,0.84149) -- 
(0.21117,0.82073) -- 
(0.23756,0.80058) -- 
(0.26396,0.78118) -- 
(0.29036,0.76178) -- 
(0.31675,0.74329) -- 
(0.34645,0.72516) -- 
(0.37284,0.70708) -- 
(0.39924,0.68995) -- 
(0.42564,0.67298) -- 
(0.45203,0.65601) -- 
(0.47843,0.64025) -- 
(0.50482,0.62449) -- 
(0.53122,0.60873) -- 
(0.55762,0.59388) -- 
(0.58401,0.57943) -- 
(0.61041,0.56473) -- 
(0.6368,0.55054) -- 
(0.6632,0.5372) -- 
(0.6896,0.52402) -- 
(0.71599,0.51053) -- 
(0.74239,0.4978) -- 
(0.76878,0.48563) -- 
(0.79518,0.47356) -- 
(0.82158,0.46143) -- 
(0.84797,0.44977) -- 
(0.87437,0.43885) -- 
(0.90076,0.42794) -- 
(0.92716,0.41703) -- 
(0.95356,0.40612) -- 
(0.97995,0.39612) -- 
(1.0063,0.38521) -- 
(1.0327,0.37551) -- 
(1.0591,0.36581) -- 
(1.0855,0.35611) -- 
(1.1119,0.34717) -- 
(1.1383,0.33869) -- 
(1.1647,0.33025) -- 
(1.1911,0.32172) -- 
(1.2175,0.31323) -- 
(1.2439,0.30474) -- 
(1.2703,0.29717) -- 
(1.2967,0.28989) -- 
(1.3231,0.28277) -- 
(1.3495,0.27529) -- 
(1.3759,0.26807) -- 
(1.4023,0.2608) -- 
(1.4287,0.25352) -- 
(1.4551,0.24716) -- 
(1.4815,0.2411) -- 
(1.5079,0.23504) -- 
(1.5343,0.22913) -- 
(1.5607,0.22291) -- 
(1.5871,0.21685) -- 
(1.6135,0.21079) -- 
(1.6399,0.20488) -- 
(1.6695,0.19922) -- 
(1.6959,0.19442) -- 
(1.7223,0.18957) -- 
(1.7487,0.18472) -- 
(1.7751,0.17988) -- 
(1.8015,0.17503) -- 
(1.8279,0.17018) -- 
(1.8543,0.16533) -- 
(1.8807,0.16048) -- 
(1.9071,0.15563) -- 
(1.9335,0.15093) -- 
(1.9599,0.1473) -- 
(1.9863,0.14366) -- 
(2.0127,0.14002) -- 
(2.0391,0.13654) -- 
(2.0655,0.13275) -- 
(2.0919,0.12911) -- 
(2.1183,0.12547) -- 
(2.1447,0.12184) -- 
(2.1711,0.1182) -- 
(2.1975,0.11456) -- 
(2.2239,0.11093) -- 
(2.2503,0.10729) -- 
(2.2767,0.10365) -- 
(2.3031,0.10002) -- 
(2.3294,0.097591) -- 
(2.3558,0.095166) -- 
(2.3822,0.092741) -- 
(2.4086,0.090317) -- 
(2.435,0.087892) -- 
(2.4614,0.085467) -- 
(2.4878,0.083043) -- 
(2.5142,0.080618) -- 
(2.5406,0.078143) -- 
(2.567,0.075769) -- 
(2.5934,0.073344) -- 
(2.6198,0.07092) -- 
(2.6462,0.068495) -- 
(2.6726,0.066071) -- 
(2.699,0.063646) -- 
(2.7254,0.061221) -- 
(2.7518,0.058797) -- 
(2.7782,0.056372) -- 
(2.8046,0.053948) -- 
(2.831,0.051523) -- 
(2.8574,0.049705) -- 
(2.8838,0.048492) -- 
(2.9102,0.04728) -- 
(2.9366,0.046068) -- 
(2.9663,0.044855) -- 
(2.9926,0.043643) -- 
(3.019,0.042431) -- 
(3.0454,0.041218) -- 
(3.0718,0.040006) -- 
(3.0982,0.038794) -- 
(3.1246,0.037632) -- 
(3.151,0.036369) -- 
(3.1774,0.035157) -- 
(3.2038,0.033945) -- 
(3.2302,0.032732) -- 
(3.2566,0.03152) -- 
(3.283,0.030308) -- 
(3.3094,0.029095) -- 
(3.3358,0.027883) -- 
(3.3622,0.026671) -- 
(3.3886,0.025458) -- 
(3.415,0.024297) -- 
(3.4414,0.023034) -- 
(3.4678,0.021821) -- 
(3.4942,0.020609) -- 
(3.5206,0.019397) -- 
(3.547,0.018185) -- 
(3.5734,0.016972) -- 
(3.5998,0.01576) -- 
(3.6262,0.014396) -- 
(3.6525,0.013184) -- 
(3.6789,0.011972) -- 
(3.7053,0.010759) -- 
(3.7317,0.0095469) -- 
(3.7581,0.0083346) -- 
(3.7845,0.0071223) -- 
(3.8109,0.00591) -- 
(3.8373,0.0046977) -- 
(3.8637,0.0034854) -- 
(3.8901,0.0022731) -- 
(3.9165,0.0010608) -- 
(3.9429,0);
\node [above right] at (0,0) {(b)};
\end{tikzpicture}
\end{tabular}
\begin{tikzpicture}[scale=2.5]
\draw (0,0) -- (2.6,0);
\draw (0,0) -- (0,1);
\draw (0,1) -- (0.35,0.74) -- (1.1,0.37) -- (2.6,0);
\node [above right] at (0,0) {(c)};
\draw[dotted] (0.35,0) -- (0.35,0.74);
\draw[dotted] (1.1,0) -- (1.1,0.37);
\end{tikzpicture}
\caption{The three cases of limit shapes in Theorem~\ref{thm:stable-fractional}: (a) triangular, (b) exponential, and (c) interpolating with $J$ linear sections, here illustrated for $J=3$.\label{fig:shapes}}
\end{figure}

First, recall from Section~\ref{sec:downscaling} that the scaling factor we employ is $a_n = n/\lambda_1 = \frac{1}{q_n}$. Thus, if $q_n$ is bounded away from zero, then the scaling $\frac{1}{q_n}$ is bounded and hence cannot transform the jumpy boundary diagrams into a smooth limit shape. Therefore we shall require
\[
q_n \to 0 \text{ as }n\to\infty.
\]

\begin{theorem}\label{thm:stable-fractional}
There are three cases for limit shapes of stable configurations of $q$-proportion Bulgarian solitaire, depending on the asymptotic behavior of $n q_n^2$:
\begin{itemize}
\item[(a)] In case $n q_n^2\rightarrow 0$, there is a triangular limit shape. Under the standard scaling we apply, by which the height of the diagram is scaled to 1, the downward slope of the triangle will be $1/2$. 
\item[(b)] In case $n q_n^2\rightarrow \infty$, there is an exponential limit shape. 
\item[(c)] Interpolating between the two previous cases is the case $n q_n^2\rightarrow C>0$. Define $z>0$ by the equation 
\[
2C = \frac{z^2+\lceil z \rceil^2}{\lceil z \rceil} - \sum_{i=0}^{\lceil z \rceil-1}\frac{1}{\lceil z \rceil-i}
\]
and set $A_0 = \frac{z}{C}\frac{1+z-\lceil z \rceil}{\lceil z \rceil}$ and
$A_k = \frac{z}{C}\frac{1}{\lceil z \rceil-k}$ for $1\le k\le \ceilz - 1$.
%
%
The limit shape approximates the exponential shape using $Z:=\lceil z\rceil$ linear sections such that the first section has width $A_0$ and every subsequent section, numbered $k=1,2,\dots,Z-1$, has width $A_k$. The slope of the $i$th section is $\frac{C(Z-i)}{z^2}$ for all $i=0,1,\dots,Z-1$.
\end{itemize}
\end{theorem}
\begin{proof}
(a) In case $n q_n^2\rightarrow 0$ we shall see that the effect of rounding dominates in a move from the stable configuration. Specifically, for all sufficiently large $n$ we have $q_n\sqrt{2n}<1$ and hence $\lceil q_n h \rceil  = 1$ for all $0<h<\sqrt{2n}$. Now, assume that the largest pile of a configuration is $\lambda_1 \approx \sqrt{2n}$. Then the size of this pile will decrease by $1$ card per move until after $\lceil \lambda_1\rceil \approx \sqrt{2n}$ moves the pile has disappeared. By Lemma~\ref{lm:stable} the corresponding triangular configuration $\overline{\lambda}$ is stable. The number of cards in this configuration is approximately $n$, which confirms that a stable configuration of $n$ cards (if it exists) will indeed have a largest pile of size  $\approx \sqrt{2n}$. 
Downscaling (where the vertical scaling factor is $1/\lambda_1=1/\sqrt{2n}$ and horizontal scaling factor $\lambda_1/n = \sqrt{2/n}$) yields a boundary diagram $\partial\overline{\lambda}$ where column $k$ has height $1-(k-1)/\sqrt{2n}$, $k=1,2,\dotsc$. Since the width of each column is $\sqrt{2/n}$, the boundary diagram function for $\overline{\lambda}$ is $\partial\overline{\lambda}(x) = 1-\lfloor x/\sqrt{2/n} \rfloor / \sqrt{2n} \to \infty$ as $n\to\infty$. By definition \eqref{eq:stable_limitshape_def}, the proposed limit shape is $y=1-x/2$ for $x\ge 0$.

(b) In case $n q_n^2\rightarrow \infty$  the effect of rounding turns out to be negligible in a move from the stable configuration. Let $\lambda_1$ denote the largest pile in the stable configuration and let $m$ denote the number of non-empty piles. In view of Lemma~\ref{lm:stable} it follows from sequential application of Lemma~\ref{lm:new-pile} and Lemma~\ref{lm:number-of-piles} that 
\[
\frac{\lambda_1}{n q_n} = 1 + O\left(\frac{m}{n q_n}\right) = 1 + O\left(\frac{\ln(q_n \lambda_1) + 1}{n q_n^2}\right).
\]
Under the assumption $n q_n^2\rightarrow \infty$ it is easy to see that the second term, which estimates the total effect of rounding after downscaling, will tend to zero. With no rounding, pile sizes decrease geometrically with decay factor $1-q_n$. Thus, after downscaling the stable configuration is asymptotically equal to a configuration $\overline{\lambda}$ with a first pile of size 1 and subsequent piles of size $1-q_n, (1-q_n)^2, \dotsc$. Since the horizontal scaling factor is $\lambda_1/n=q_n$, the width of each column in the rescaled boundary diagram $\partial\overline{\lambda}$ is $q_n$. Thus $\partial\overline{\lambda}(x) = (1-q_n)^{\lfloor x/q_n \rfloor}$ for $x\ge 0$. As $n\to\infty$, $q_n\to 0$ and $\partial\overline{\lambda}(x) \to e^{-x}$, proving the proposed limit shape.

(c) For the remaining case, the crucial observation is that the rate by which a pile melts away depends on how the pile size relates to multiples of $1/q_n$. Any pile size can be expressed on the form $y/q_n$ for some $y>0$. From a pile of that size, a move will take away the amount $\lceil y \rceil$. Thus, a pile starting at a size of $z/q_n$ will initially melt away at a slope of $Z=\lceil z \rceil$ per move for
$\Bigl\lceil \frac{1+z-Z}{Zq_n}  \Bigr\rceil$
moves, i.e.\ until the pile size reaches the threshold $(Z-1)/q_n$. At this point the slope decreases to $Z-1$ per move for
$\Bigl\lceil \frac{1}{q_n(Z-1)} \Bigr\rceil$
moves until the pile size reaches the next threshold, $(Z-2)/q_n$, etc. This pattern ends with a section of slope 1 per move for
$\lceil 1/ q_n \rceil$ moves. See Figure~\ref{fig:interpolate}. By Lemma~\ref{lm:stable} this sequence of pile sizes constitutes a stable configuration $\Lambda$. It is elementary, although somewhat tedious, to verify that the total amount of cards $n'$ in this configuration is 
\[
n' = \sum_{i=0}^{Z-1}
\left[ \frac{zB_i}{q_n} - (B_i-1)(Z-i)\Big(\frac{B_i}{2} + \sum_{j=i+1}^{Z-1}B_j\Big)
\right]
\]
where $B_0 = \lceil \frac{1+z-Z}{qZ} \rceil$ and $B_k = \lceil \frac{1}{q(Z-k)} \rceil$
for $1\le k\le Z - 1$ are the lengths of the $Z$ sections with constant slope in the configuration $\Lambda$.
Under the assumptions $nq_n^2\to C$ and $q_n\to 0$ as $n\to\infty$, it is straightforward (but again somewhat tedious) to verify that, as $n\to\infty$, we have
\[
q_n^2 n' \to \frac{1}{2} \left( \frac{z^2+Z^2}{Z} - \sum_{i=0}^{Z-1}\frac{1}{Z-i} \right) = C
\]
where the equality comes from using the value of $z$ defined in the theorem. 
The actual total amount of cards that we play with is $n$. 
%
%
Thus, the two amounts of cards $n$ and $n'$ are asymptotically equal under the assumption $n q_n^2\rightarrow C$, in which case the actual stable configuration is asymptotically equal to the stable configuration $\Lambda$ described above.

Let $A_i$ be the length of the $i$th section, $0 \le i \le Z-1$, after downscaling $\Lambda$ with our standard scaling factors (horizontal scaling by $\frac{z}{nq_n}$ and vertical scaling by $\frac{q_n}{z}$). Then
\begin{align*}
A_0 &= \frac{z}{nq_n} \left\lceil \frac{1+z-Z}{q_n Z} \right\rceil \to \frac{z(1+z-Z)}{CZ} \text{ and} \\
A_k &= \frac{z}{nq_n} \left\lceil \frac{1}{q_n(Z-k)} \right\rceil \to \frac{z}{C(Z-k)},\; 1 \le k \le Z-1
\end{align*}
as $n\to\infty$. The proposed slopes of the sections follow immediately. Analogously to the proof in case (a), it follows that the above describes the limit shape.
\end{proof}
\begin{figure}[ht]
\centering
\begin{tikzpicture}[scale=0.15,>=latex']
\draw [thick,<->] (0,43) -- (0,0) -- (73,0);

\draw[thick] (0,40.0) -- (9,30.0);
\draw[thick] (9,30.0) -- (21,20.0);
\draw[dashed] (21,20.0) -- (38,10.0);
\draw[thick] (38,10) -- (71.3,0);

\draw [<->] (0,-2) -- (9,-2);
\node [below] at (4.5,-2)
{$\Bigl\lceil \frac{1+z-Z}{q_n Z} \Bigr\rceil$};
\draw [<->] (9,-2) -- (21,-2);
\node [below] at (15,-2)
{$\Bigl\lceil \frac{1}{q_n(Z-1)} \Bigr\rceil$};
\draw [<->] (38,-2) -- (71.3,-2);
\node [below] at (54.1,-2) {$\lceil 1/q_n \rceil$};

\draw (-1,0) -- (0,0);
\node [left] at (-1,0) {$0$};
\draw (-1,10) -- (0,10);
\node [left] at (-1,10) {$\frac{1}{q_n}$};
\draw (-1,20) -- (0,20);
\node [left] at (-1,20) {$\frac{Z-2}{q_n}$};
\draw (-1,30) -- (0,30);
\node [left] at (-1,30) {$\frac{Z-1}{q_n}$};
\draw (-1,40) -- (0,40);
\node [left] at (-1,40) {$\frac{z}{q_n}$};

\draw[dotted] (9,-2) -- (9,30);
\draw[dotted] (21,-2) -- (21,20);
\draw[dotted] (38,-2) -- (38,10);

\draw (0,0) -- (0,-1);
\draw[dotted] (0,30) -- (9,30);
\draw[dotted] (0,20) -- (21,20);
\draw[dotted] (0,10) -- (38,10);

\node [right] at (4.5,35) {slope $Z$};
\node [right] at (15,25.4) {slope $Z-1$};
\node [right] at (54.1,5.8) {slope $1$};

\end{tikzpicture}
\caption{The stable configuration $\Lambda$ in case (c) of the proof of Theorem~\ref{thm:stable-fractional}.\label{fig:interpolate}}
\end{figure}
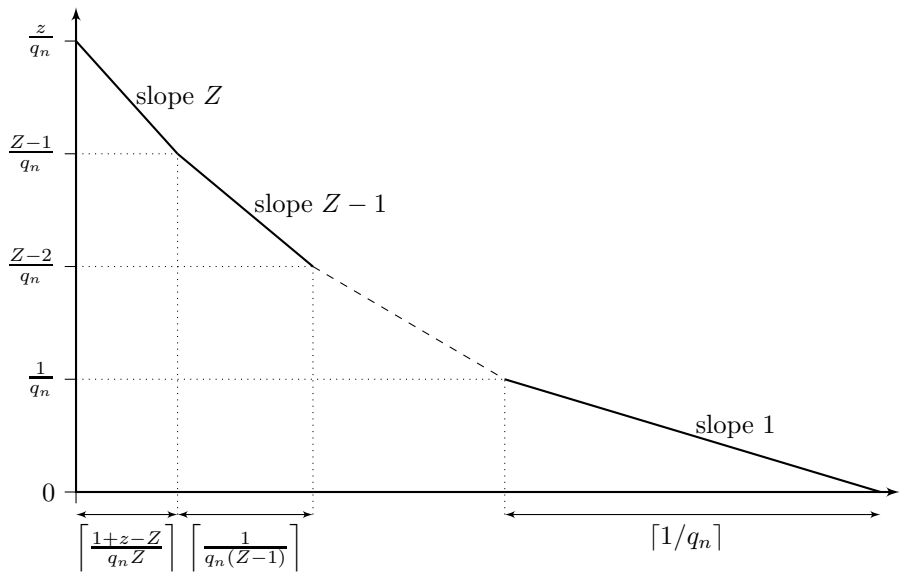

\section{Limit shapes of recurrent configurations of $q_n$-proportion Bulgarian solitaire}\label{sec:q-recurrent}
Although we have not been able to prove Conjecture~\ref{conj:stable-limit-is-limit} for general well-behaved $\sigma$-solitaire, we can prove the conjecture in the special cases of the two main regimes of $q_n$-proportion solitaire.

\begin{lemma}\label{lm:normal}
After $n$ moves of $q_n$-proportion solitaire on $n$ cards there are at most $2\sqrt{n}$ nonempty piles and the largest pile is at most of size $nq+O(\sqrt{n})$. 
\end{lemma}
\begin{proof}
Every nonempty pile decreases by at least one card in each move. As all pile sizes are bounded by $n$,  all original piles must have died out after $n$ moves. Moreover, because there are $n$ cards in total there are always at most $\sqrt{n}$ piles of size greater than 
$\sqrt{n}$. After $\sqrt{n}$ moves all other piles will have died and $\sqrt{n}$ new piles will have been created, hence there will then be at most $2\sqrt{n}$ nonempty piles. From then on, when new piles are formed they will have size $nq+O(\sqrt{n})$, where the latter term is the contribution from the number of picked cards in each pile being rounded upwards to the closest integer. 
\end{proof}

\subsection{The limit shape of recurrent configurations in the case $q_n^2 n\rightarrow 0$}
In case $q_n^2 n\rightarrow 0$, Lemma~\ref{lm:normal} implies that after $n$ moves the largest pile size is $O(\sqrt{n})$. Then the number of picked cards in each pile is bounded by $\lceil q_n O(\sqrt{n}))\rceil$. This number equals 1 for sufficiently large $n$. From then on the solitaire is therefore equivalent to ordinary Bulgarian solitaire. As the recurrent configurations of ordinary Bulgarian solitaire are known to converge to a triangular limit shape, it follows that the recurrent configurations of $q_n$-proportion Bulgarian solitaire do too in this case.

\subsection{The limit shape of recurrent configurations in the case $q_n^2 n\rightarrow \infty$}
Finally, we shall prove that in the case $q_n^2 n\rightarrow \infty$, the recurrent configurations of $q_n$-proportion solitaire have an exponential limit shape. We do this by showing that regardless of which configuration we start at we must eventually reach configurations that are close to the exponential shape. Our proof works with piles sorted by time of creation rather than by size. Thus, as mathematical objects the configurations are then compositions rather than partitions. However, as we prove in the companion paper \cite{EJS2}, if a sequence of compositions has a decreasing limit shape then the same limit shape is obtained by the corresponding partitions.

Lemma~\ref{lm:normal} implies that after $n$ moves the largest pile is always of size $nq_n+O(\sqrt{n})=n( q_n+o(1))$ and that after an additional $2\sqrt{n}$ moves all non-empty piles will be stemming from piles of that size. At this point, let $\alpha_i$ denote the size of the pile that was created $i$ moves ago ($i=1,2,\dotsc$) and has since been decreased $i-1$ times. Thus $\alpha_i = (1-q_n)^{i-1}n(q_n+o(1)) - O(i)$ where the latter term is the contribution from rounding downward in each move.

Let $x=x(i)$, $x \ge 0$, be the $x$-coordinate of the column corresponding to $\alpha_i$ in the composition diagram of $\alpha$, and $y$ the height of this column, after downscaling:
\begin{equation}
\label{eq:xy}
x=(i-1)(q_n+o(1)), \quad
y=\frac{(1-q_n)^{i-1}n(q_n+o(1)) - O(i)}{n(q_n+o(1))}.
\end{equation}
Then $i=1+\frac{x}{q_n+o(1)}$ and
\begin{align*}
y(x) &= \frac{(1-q_n)^{x/(q_n+o(1))}n(q_n+o(1))-O( \frac{x}{q_n+o(1)} )}{n(q_n+o(1))} \\
&=(1-q_n)^{x/(q_n+o(1))} - O\left(\frac{x}{n(q_n+o(1))^2} \right) \\
&=(1-q_n)^{x/(q_n+o(1))} - o(1),
\end{align*}
where we used $q_n^2 n \rightarrow \infty$ as $n \rightarrow \infty$ in the last step.
Since ${(1- q_n)^{x/q_n} \rightarrow e^{-x}}$, we therefore have $y(x)\rightarrow e^{-x}$. 
Recall that $q_n \rightarrow 0$ as $n\to\infty$. Thus, by \eqref{eq:xy}, for a fixed $i$ we have $x=x(i)\to 0$ as $n\to\infty$. In other words, the
width of each column in the composition diagram for $\alpha$ (after downscaling) is $q_n+o(1) \rightarrow 0$.
This means that the limit $y(x)\to e^{-x}$ holds for any $x \ge 0$, which, by the definition \eqref{eq:deterministic_limitshape_def}, establishes the limit shape for 
the $q_n$-proportion Bulgarian solitaire.


Finally, thanks to the abovementioned result from \cite{EJS2}, the same limit shape is obtained when the piles of the weak compositions are reordered to form partitions.

\section{Discussion}
In this paper we have discussed limit shape results for stable and recurrent configurations. Popov \cite{Popov} studied the limit shape of the configurations drawn from the stationary distribution of a random version of Bulgarian solitaire, in which a card is picked from a pile only with probability $p$ (and independently of other piles). He found that also this random version yields a triangular limit shape, in the sense that the probability of deviations larger than some $\varepsilon>0$ tends to zero as $n$ tends to infinity. In a companion paper \cite{EJS2} to the present paper we 
study 
a random version of $q_n$-proportion Bulgarian solitaire 
and the conditions under which it has an exponential limit shape.





\end{document}